\documentclass[11pt]{amsart}

\usepackage[margin=3cm]{geometry}

\usepackage[foot]{amsaddr}
\usepackage{amsmath}
\usepackage{amssymb}
\usepackage{amsthm}
\usepackage{graphicx}
\usepackage{float}
\usepackage{color}
\usepackage{dsfont}
\usepackage{comment}
\usepackage[T1]{fontenc}
\usepackage[utf8]{inputenc}
\usepackage{tabularx,ragged2e,booktabs,caption}

\newtheorem{theorem}{Theorem}[section]

\newtheorem{lemma}[theorem]{Lemma}

\theoremstyle{definition}

\newtheorem{example}[theorem]{Example}

\title[Monotonicity of the  principal eigenvalue] 
      { Monotonicity of the  principal eigenvalue  for a linear time-periodic parabolic operator
      }

\author{Shuang Liu,\ \ Yuan Lou,\ \ Rui Peng\ \ and\ \, Maolin Zhou}

\thanks{{S. Liu}: Institute for Mathematical Sciences, Renmin University of China, Beijing 100872, PRC. Email: liushuangnqkg@ruc.edu.cn}

\thanks{{Y. Lou}: Department of Mathematics, Ohio State University, Columbus, OH 43210, USA. Email:  lou@math.ohio-state.edu}

\thanks{{R. Peng}: School of Mathematics and Statistics, Jiangsu Normal University,
Xuzhou,  Jiangsu 221116, PRC. Email:
pengrui\,$\b{}$\,seu@163.com}

\thanks{{M. Zhou}: Department of Mathematics, School of Science and Technology,
University of New England, Armidale, NSW 2341, Australia. Email:
zhouutokyo@gmail.com}


\begin{document}
\maketitle

\begin{abstract}
We investigate the effect of frequency on the principal eigenvalue
 of a time-periodic parabolic operator with Dirichlet, Robin or Neumann boundary conditions. The monotonicity and  asymptotic behaviors of the principal eigenvalue with respect to the frequency parameter are established. Our results
 prove a conjecture raised by Hutson, Michaikow and Pol\'{a}\v{c}ik \cite{Hutson2001}.
\end{abstract}
\section{ Introduction}\label{S1}
We consider the linear time-periodic parabolic eigenvalue problem
\begin{equation}\label{fr1}
\begin{cases}
\begin{array}{ll}
\tau\partial_{t}u-\mathrm{div}\left[A(x,t)\nabla u\right]-\nabla m(x,t)\cdot\nabla u+V(x,t)u=\lambda u, &x\in\Omega,t\in[0,1],\\
bu+(1-b)\left[A\nabla u\right]\cdot\mathbf{n}=0, &x\in\partial\Omega,t\in[0,1],\\
u(x,0)=u(x,1), &x\in\Omega,
\end{array}
\end{cases}
\end{equation}
where $\Omega$ is a bounded domain in $\mathbb{R}^{N}$ with  smooth boundary $\partial\Omega$, and $\mathbf{n}$ denotes the unit outward normal vector on $\partial\Omega$.
 The positive constant $\tau$ is referred as
 the frequency, and the constant $b\in[0,1]$. 
 Denote by $\mathcal{C}=\Omega\times(0,1)$ the periodicity cell. The matrix function $A\in C^{1+\sigma,1}(\overline{\mathcal{C}})$,
 and functions $m\in C^{1+\sigma,1}(\overline{\mathcal{C}})$ and $V\in C^{\sigma,1}(\overline{\mathcal{C}})$,
  with $\sigma\in (0, 1)$,  are assumed to be periodic in $t$ with  unit period.
  Furthermore, 
  $A$ is a symmetric and uniformly  elliptic  matrix field, i.e. there exist positive constants $\gamma_1$ and $\gamma_2$ such that $\gamma_1|\xi|^{2}\leq\xi^{\mathrm{T}}A(x,t)\xi\leq\gamma_2|\xi|^{2}$ holds for all $(x,t)\in\mathcal{C}$ and $\xi\in\mathbb{R}^{N}$.

Proposition 14.4 in \cite{Hess} guarantees the existence and uniqueness of the principle eigenvalue, denoted by $\lambda(\tau)$,
 of problem (\ref{fr1}), which is
real, simple and its corresponding eigenfunction can be chosen 
positive in $\Omega\times [0, 1]$.
Furthermore, $\lambda(\tau)<Re(\lambda)$ for any other eigenvalue $\lambda$ of  (\ref{fr1}).

 Problem (\ref{fr1}) arises in  connection with the nonlinear reaction-diffusion equation
\begin{equation}\label{fr_20}
\begin{cases}
 \begin{array}{ll}
 \smallskip
 \partial_{s}w=\nabla\cdot\left[A^*(x,s)\nabla w\right]+\nabla w\cdot \nabla m^*(x,s)+wf(x,s,w), & x\in\Omega, s>0,\\
 \smallskip
 bw+(1-b)\left[A^*(x,s)\nabla w\right]\cdot\mathbf{n}=0, & x\in\partial\Omega, t>0,\\
 w(x,s)=w(x,s+T),& x\in\Omega,
  \end{array}
  \end{cases}
 \end{equation}
which models various ecological and evolutionary processes  in spatio-temporally varying environments
\cite{CC2003, CL2008,CL2012, Ni2011, PZZ2018, PZ2017}.
 The matrix function $A^*$
 and functions $m^*$, $f$ are periodic in $s$ with  a common period $T>0$.
 It is natural to inquire how the temporal variability of the environment
 affects the population dynamics of \eqref{fr_20}.
The persistence of populations is closely associated with
the stability of the steady state $w=0$
for (\ref{fr_20}), which in turn is determined by the sign of the principal eigenvalue,
denoted by $\lambda^*(T)$, of the linear eigenvalue problem
\begin{equation*}
\begin{cases}
 \begin{array}{ll}
 \smallskip
 \partial_{s}\varphi-\nabla\cdot\left[A^*(x,s)\nabla \varphi\right]-\nabla \varphi\cdot \nabla m^*(x,s)
 -f(x,s,0) \varphi=\lambda\varphi, & x\in\Omega, s>0,\\
 \smallskip
 b\varphi+(1-b)\left[A^*(x,s)\nabla \varphi\right]\cdot\mathbf{n}=0, & x\in\partial\Omega, s>0,\\
 \varphi(x,s)=\varphi(x,s+T),& x\in\Omega.
  \end{array}
  \end{cases}
 \end{equation*}

Set
$\tau=1/T, t=s/T, A(x,t)=A^*(x,s), m(x,t)=m^*(x, s), V(x,t)=-f(x,s,0),
$ and $u(x,t)=\varphi(x,s).
$
Then $u(x,t)$ satisfies problem (\ref{fr1}),
and $\lambda(\tau)=\lambda^*(T)$. Hence,
determining the stability of $w=0$ is reduced
to understanding the sign of $\lambda(\tau)$.
The goal of this paper is to study the dependence of  $\lambda(\tau)$ on the frequency $\tau$.

Given any function $p(x,t)$, which is $1$-periodic in time,  set
$$\hat{p}(x):=\int_0^1p(x,s)\mathrm{d}s.$$

For $A(x,t)\equiv \hat{A}(x)$  and $\nabla m(x,t)\equiv \nabla\hat{m}(x)$, it was shown in \cite{Hutson2000} that $\lambda(\tau)\leq\lim\limits_{\tau\rightarrow\infty}\lambda(\tau)$ for all $\tau>0$, i.e.  $\lambda(\tau)$ attains its maximum at $\tau=\infty$.
  It is natural to ask whether $\lambda(\tau)$ is increasing in $\tau$.
The monotonicity of $ \lambda(\tau)$ seems to be open \cite{Nadin2009}, even  for  the case $\nabla m=0$ as conjectured by
  Hutson et. al \cite{Hutson2001}. We now answer this conjecture positively as follows:
\begin{theorem}\label{FR_thm 1}
Assume $\nabla m=0$. Then  $\lambda(\tau)$ 
 is non-decreasing in
$\tau>0$. If further assume  $A(x,t)\equiv \hat{A}(x)$, then the following assertions hold:
\begin{itemize}
\item[(i)]  If  $V= \hat{V}(x)+g(t)$ for some $1$-periodic function $g(t)$, then
 $\lambda(\tau)$ is constant for  $\tau>0$;
 \item[(ii)] Otherwise $\frac{\partial\lambda}{\partial\tau}(\tau)>0$ for every $\tau>0$.
\end{itemize}
\end{theorem}
The monotonicity of $\lambda(\tau)$ in  Theorem \ref{FR_thm 1}
may fail  if $\nabla m\neq0$; See Example \ref{Ex1} in Sect. \ref{S3}.
Our next result concerns the case $A=I_{N\times N}$ and $\nabla m(x,t)=\nabla\hat{m}(x)$.

\begin{theorem}\label{FR_thm 1.1}
Assume $A(x,t)=DI_{N\times N}$ for some constant $D>0$ and $\nabla m(x,t)=\nabla\hat{m}(x)$. Then $\lambda(\tau)$ is non-decreasing in
 $\tau>0$ and  the following assertions hold:
\begin{itemize}
\item[(i)]  If  $V= \hat{V}(x)+g(t)$ for some $1$-periodic function $g(t)$, then
 $\lambda(\tau)$ is constant for $\tau>0$;
 \item[(ii)] Otherwise $\frac{\partial\lambda}{\partial\tau}(\tau)>0$ for every  $\tau>0$.
\end{itemize}
\end{theorem}

 Interpreting them in the context of  spatio-temporal variation of environmental resources,
 Theorems   \ref{FR_thm 1}  and \ref{FR_thm 1.1}  suggest that if the species disperse by random diffusion
  and/or by advection along  
  some gradient $\nabla m$,
 increasing the temporal variation of the resources tends to favor the persistence of populations.  The condition $\nabla m(x,t)=\nabla\hat{m}(x)$ in Theorem \ref{FR_thm 1.1} is necessary;
  See 
  Example \ref{Ex2} in Sect. \ref{S3}.

We next state the asymptotic behaviors of $\lambda(\tau)$
 for sufficiently small or large $\tau$:
\begin{theorem}\label{FR_thm 2}
The following
assertions hold:
\begin{itemize}
\item[(i)]  For each fixed $t\in [0, 1]$, denote by $\lambda^0(t)$ the principal eigenvalue of the linear problem
\begin{equation}\label{fr2}
\begin{cases}
\begin{array}{ll}
-\mathrm{div}\left[A(x,t)\nabla \varphi\right]-\nabla m(x,t)\cdot\nabla\varphi+V(x,t)\varphi=\lambda(t)\varphi, &x\in\Omega,\\
b\varphi+(1-b)\left[A\nabla \varphi\right]\cdot\mathbf{n}=0,&x\in\partial\Omega,
\end{array}
\end{cases}
\end{equation}
then
$\lim_{\tau\rightarrow0}\lambda(\tau)=\int_{0}^{1}\lambda^0(s)\mathrm{d}s;
$
 \item[(ii)]
 Denote by $\lambda^{\infty}$ the principal eigenvalue of the linear problem
 \begin{equation}\label{fr3}
\begin{cases}
\begin{array}{ll}
-\mathrm{div}\left[\hat{A}(x)\nabla \varphi\right]-\nabla \hat{m}(x)\cdot\nabla\varphi+\hat{V}(x)\varphi=\lambda\varphi, &x\in\Omega,\\
b\varphi+(1-b)\left[\hat{A}\nabla \varphi\right] \cdot\mathbf{n}=0,&x\in\partial\Omega,
\end{array}
\end{cases}
\end{equation}
then
 $\lim_{\tau\rightarrow\infty}\lambda(\tau)=\lambda^{\infty}.
 $
\end{itemize}
\end{theorem}

 Part (i) in  Theorem \ref{FR_thm 2}  appears to be new,
  while part (ii) is due to Nadin \cite{Nadin2009} for the space-time periodic case. Theorems \ref{FR_thm 1}, \ref{FR_thm 1.1} and  \ref{FR_thm 2} 
 imply that if $\nabla m(x,t)=\nabla\hat{m}(x)$,
$\int_{0}^{1}\lambda^0(s)\mathrm{d}s\leq \lambda(\tau)\leq \lambda^\infty$
for any $\tau$, 
and the equality holds
 if and only if  $V$ has the form of $\hat{V}(x)+g(t)$. Here the estimate $\lambda(\tau)\leq\lambda^{\infty}$
  agrees with the result in \cite{Hutson2000}, while 
  $\lambda(\tau)\geq\int_{0}^{1}\lambda^0(s)\mathrm{d}s$ is new, 
    even in the simplest scenario when $A=I_{N\times N}$ and $\nabla m=0$;
    See \cite{Hutson2001} for  numerical results for this case.

The rest of the paper is organized as follows: In Sect. \ref{S2},
   Theorems   \ref{FR_thm 1} and    \ref{FR_thm 1.1}
  are proved.
Two examples are presented in Sect. \ref{S3} to complement Theorems  \ref{FR_thm 1} and    \ref{FR_thm 1.1}.
Sect. \ref{S4} is devoted to the proof of  Theorem   \ref{FR_thm 2}.

\section{ Monotonicity of $\lambda(\tau)$}\label{S2}
This section is devoted to the proofs of  Theorems \ref{FR_thm 1} and   \ref{FR_thm 1.1}. The proofs are based upon some functional, which was first introduced in \cite{LL2018} for an  elliptic eigenvalue problem.

\subsection{The case $\nabla m=0$}\label{S3.1}
In this subsection, we consider the scenario when $\nabla m=0$ and prove Theorem \ref{FR_thm 1}. In this case, problem (\ref{fr1}) becomes
\begin{equation}\label{fr21}
\begin{cases}
\begin{array}{ll}
\smallskip
L_\tau u:=\tau\partial_{t}u-\mathrm{div}\left[A(x,t)\nabla u\right]+V(x,t)u=\lambda(\tau)u, &x\in\Omega,t\in[0,1],\\
\smallskip
bu+(1-b)\left[A\nabla u\right]\cdot\mathbf{n}=0, &x\in\partial\Omega,t\in[0,1],\\
u(x,0)=u(x,1), &x\in\Omega,
\end{array}
\end{cases}
\end{equation}
where we denote $u_\tau$ as a positive eigenfunction associated with  $\lambda(\tau)$. Furthermore,  consider the adjoint problem of (\ref{fr21}), i.e. 
\begin{equation}\label{fr112}
\begin{cases}
\begin{array}{ll}
\smallskip
L^{*}_\tau v:=-\tau\partial_{t}v-\mathrm{div}\left[A(x,t)\nabla v\right]+V(x,t)v=\lambda(\tau)v, &x\in\Omega,t\in[0,1],\\
\smallskip
bv+(1-b)\left[A\nabla v\right]\cdot\mathbf{n}=0, &x\in\partial\Omega,t\in[0,1],\\
v(x,0)=v(x,1), &x\in\Omega.
\end{array}
\end{cases}
\end{equation}

Let $v_\tau$ be  a positive eigenfunction of \eqref{fr112} corresponding to $\lambda(\tau)$. We normalize $u_\tau$ and $v_\tau$ such that $\int_{\mathcal{C}}u^2_\tau=\int_{\mathcal{C}}u_\tau v_\tau=1$ for any $\tau>0$.

 For  $b\in[0,1]$,  define set $\mathbb{S}_b$ by
 \begin{equation*}
\mathbb{S}_b=\left\{
\begin{array}{ll}
\medskip
\zeta \in  C^{2,1}(\mathcal{C})\cap C^{1,1}(\overline{\mathcal{C}}):  &\zeta(x,0)=\zeta(x,1)~\mathrm{in}~\Omega,\\ &b\zeta+(1-b)[A\nabla \zeta]\cdot\mathbf{n}=0 ~\mathrm{on} ~\partial\Omega \times[0,1]
\end{array}\right\}.
 \end{equation*}

Define functional $J_\tau$ by
 \begin{equation}\label{fr7}
   J_\tau(\zeta)=\int_{\mathcal{C}}u_\tau v_\tau\left(\frac{ L_{\tau}\zeta}{\zeta}\right)\mathrm{d}x\mathrm{d}t, \quad\zeta\in\mathbb{S}^{0}_b,
 \end{equation}
where
\begin{equation*}
\mathbb{S}^{0}_b=
 \begin{cases}
 \begin{split}
 &\{\zeta \in \mathbb{S}_b:  \zeta>0 ~\mathrm{in}~\overline{\mathcal{C}}\}, &0\leq b<1,\\
 &\{\zeta \in \mathbb{S}_1:  \zeta>0 ~\mathrm{in}~\mathcal{C}, \nabla\zeta\cdot\mathbf{n}<0~\mathrm{on} ~\partial\Omega \times[0,1]\},&b=1.
  \end{split}
 \end{cases}
 \end{equation*}
 Clearly, for any $\tau>0$, $u_\tau, v_\tau\in \mathbb{S}^{0}_b$ and $ J_\tau$ is well defined on the cone $\mathbb{S}^{0}_b$. The following property of  $J_\tau$ turns out to be  crucial in establishing Theorem \ref{FR_thm 1}.
\begin{lemma}\label{Flem2.1}
For any  $\zeta\in\mathbb{S}^{0}_b$, we have
\begin{equation}\label{fr8}
   J_\tau(u_\tau)-J_\tau(\zeta)=\int_{\mathcal{C}}u_\tau v_\tau \left[\nabla\log\left(\frac{\zeta}{u_\tau}\right)\right]\cdot\left[ A\nabla\log\left(\frac{\zeta}{u_\tau}\right)\right].
\end{equation}
\end{lemma}
\begin{proof}
By the definition of $J_\tau$, we observe that, for every $\zeta\in\mathbb{S}^{0}_b$,
\begin{equation}\label{fr9}
\begin{split}
J_\tau(\zeta)=&\tau\int_{\mathcal{C}}  u_\tau  v_\tau\left(\frac{\zeta_t}{\zeta}\right)-\int_{\mathcal{C}} u_\tau  v_\tau \left[\frac{\mathrm{div}(A\nabla \zeta)}{\zeta}\right]+\int_{\mathcal{C}} u_\tau  v_\tau V\\
=&\tau\int_{\mathcal{C}}  u_\tau  v_\tau\partial_t \log \zeta-\int_{0}^{1}\!\!\!\int_{\partial\Omega} u_\tau  v_\tau \left[A\nabla\log\zeta\right]\cdot\mathbf{n}+\int_{\mathcal{C}}\nabla\left( u_\tau  v_\tau\right)\cdot \left[A\nabla\log\zeta\right]
\\&-\int_{\mathcal{C}} u_\tau  v_\tau\Big[(\nabla\log\zeta)\cdot\left( A\nabla\log\zeta\right)\Big]+\int_{\mathcal{C}} u_\tau  v_\tau V.\\
\end{split}
\end{equation}
In what follows, we need to distinguish two cases: $0\leq b<1$ and $b=1$.\\

\noindent {\bf Case 1:  $0\leq b<1$.}
First, we claim that $u_\tau$ is  a critical point of $J_\tau$ in the sense that
\begin{equation}\label{fr_DJ}
\begin{array}{l}
\mathbf{D}J_\tau(u_\tau)\varphi=0\,\, \text{ for all} \,\,\varphi\in\mathbb{S}_b,
  \end{array}
 \end{equation}
 where  $\mathbf{D}J_\tau(u_\tau)$ is the Fr\'{e}chet derivative of  $J_\tau$ at the point $u_\tau\in\mathbb{S}^{0}_b$.

To prove (\ref{fr_DJ}), we first observe that, by $\zeta\in\mathbb{S}^{0}_b$,
\begin{equation*}
-\int_{0}^{1}\!\!\!\int_{\partial\Omega} u_\tau  v_\tau \left[A\nabla\log\zeta\right]\cdot\mathbf{n}=\frac{b}{1-b}\int_{0}^{1}\!\!\!\int_{\partial\Omega} u_\tau  v_\tau,
\end{equation*}
which is independent of $\zeta$ in the current case. 
Making use of this fact and  (\ref{fr9}), for any $\varphi\in\mathbb{S}_b$, we have
\begin{equation}\label{fr10}
\begin{split}
\mathbf{D}J_\tau(\zeta)\varphi=&\tau\int_{\mathcal{C}}  u_\tau  v_\tau\partial_t\left(\frac{\varphi}{\zeta}\right)-2\int_{\mathcal{C}} u_\tau v_\tau (\nabla\log\zeta)\cdot\left[ A\nabla\left(\frac{\varphi}{\zeta}\right)\right]\\
&+\int_{\mathcal{C}}\Big[A\nabla( u_\tau v_\tau)\Big]\cdot\nabla\left(\frac{\varphi}{\zeta}\right).
\end{split}
\end{equation}
Through straightforward calculations, we further have
 \begin{equation*}
\begin{split}
&\mathbf{D}J_\tau( u_\tau)\varphi\\
=&\tau\int_{\mathcal{C}}  u_\tau  v_\tau\partial_t\left(\frac{\varphi}{ u_\tau}\right)-2\int_{\mathcal{C}} u_\tau  v_\tau (\nabla\log u_\tau)\cdot\left[ A\nabla\left(\frac{\varphi}{ u_\tau}\right)\right]+\int_{\mathcal{C}}\Big[A\nabla( u_\tau  v_\tau)\Big]\cdot\nabla\left(\frac{\varphi}{ u_\tau}\right)\\
=&-\tau\int_{\mathcal{C}} \left(\frac{\varphi}{ u_\tau}\right)\partial_t( u_\tau  v_\tau)-2\int_{0}^{1}\!\!\!\int_{\partial\Omega} \left(\frac{\varphi v_\tau}{ u_\tau}\right)A\nabla u_\tau\cdot\mathbf{n}+2\int_{\mathcal{C}}\left(\frac{\varphi}{ u_\tau}\right)\nabla\cdot\Big[ v_\tau A \nabla  u_\tau\Big]\\
&+\int_{0}^{1}\!\!\!\int_{\partial\Omega}\left(\frac{\varphi}{ u_\tau}\right)A\nabla( u_\tau v_\tau)\cdot\mathbf{n}-\int_{\mathcal{C}}\left(\frac{\varphi}{ u_\tau}\right)\nabla\cdot\Big[A\nabla( u_\tau  v_\tau)\Big]\\
=&-\tau\int_{\mathcal{C}} \left(\frac{\varphi}{ u_\tau}\right)\partial_t( u_\tau  v_\tau)+2\int_{\mathcal{C}}\left(\frac{\varphi}{ u_\tau}\right) \Big\{\nabla v_\tau\cdot\left[A \nabla  u_\tau\right]+ v_\tau\mathrm{div}\left[A \nabla  u_\tau\right]\Big\}\\
&-\int_{\mathcal{C}}\left(\frac{\varphi}{ u_\tau}\right) \Big\{ v_\tau\mathrm{div}\left[A\nabla  u_\tau\right]+2\nabla  v_\tau\cdot \left[A\nabla  u_\tau\right]+ u_\tau\mathrm{div}\left[A\nabla  v_\tau\right]\Big\}\\
=&-\tau\int_{\mathcal{C}} \left(\frac{\varphi}{ u_\tau}\right)\partial_t( u_\tau  v_\tau)+\int_{\mathcal{C}}\left(\frac{\varphi v_\tau}{ u_\tau}\right)\mathrm{div}\left[A \nabla  u_\tau\right]-\int_{\mathcal{C}}\varphi\mathrm{div}\left[A\nabla  v_\tau\right]\\
=&-\int_{\mathcal{C}}\left(\frac{ v_\tau}{ u_\tau}\right)\Big\{\tau\partial_t u_{\tau}-\mathrm{div}\left[A \nabla u_\tau\right]\Big\}\varphi+\int_{\mathcal{C}}\Big\{ -\tau \partial_tv_{\tau}-\mathrm{div}\left[A\nabla  v_\tau\right]\Big\}\varphi.\\
\end{split}
\end{equation*}
In the above, the boundary integrals vanish due to the boundary conditions of $u_\tau $ and $v_\tau $.

By $ L_{\tau}u_\tau =\lambda(\tau)u_\tau $ and $ L^{*}_{\tau} v_\tau =\lambda(\tau)v_\tau $, we obtain
\begin{equation*}
\begin{split}
\mathbf{D}J_\tau(u_\tau )\varphi=&-\int_{\mathcal{C}}\left(\frac{v_\tau }{u_\tau }\right)\left[\lambda(\tau)u_\tau -Vu_\tau \right]\varphi +\int_{\mathcal{C}}\left[\lambda(\tau)v_\tau -Vv_\tau \right]\varphi=0,
\end{split}
\end{equation*}
and (\ref{fr_DJ}) thus follows.

We now proceed to prove formula (\ref{fr8}) through some  tedious manipulations. With the help of (\ref{fr9}), direct  calculation shows
\begin{equation*}
\begin{split}
&J_\tau( u_\tau)-J_\tau(\zeta)\\
=&-\tau\int_{\mathcal{C}}  u_\tau  v_\tau\partial_t\log \left(\frac{\zeta}{ u_\tau}\right)-\int_{\mathcal{C}} u_\tau  v_\tau(\nabla\log  u_\tau)\cdot\left[ A\nabla\log  u_\tau\right]+\int_{\mathcal{C}} u_\tau  v_\tau(\nabla\log \zeta)\cdot\left[ A\nabla\log \zeta\right] \\
&-\int_{\mathcal{C}}\nabla( u_\tau  v_\tau)\cdot\left[A\nabla\log \left(\frac{\zeta}{ u_\tau}\right)\right]\\
=&-\tau\int_{\mathcal{C}}  u_\tau  v_\tau\partial_t\log \left(\frac{\zeta}{ u_\tau}\right)+\int_{\mathcal{C}} u_\tau v_\tau \Big[\nabla\log\left(\zeta u_\tau\right)\Big]\cdot\left[A \nabla\log\left(\frac{\zeta}{ u_\tau}\right)\right] \\&-\int_{\mathcal{C}}\nabla( u_\tau  v_\tau)\cdot\left[A\nabla\log \left(\frac{\zeta}{ u_\tau}\right)\right]\\
=&-\tau\int_{\mathcal{C}}  u_\tau  v_\tau\partial_t\log \left(\frac{\zeta}{ u_\tau}\right)+\int_{\mathcal{C}} u_\tau v_\tau \left[\nabla\log\left(\frac{\zeta}{ u_\tau}\right)+2\nabla\log  u_\tau\right]\cdot\left[A \nabla\log\left(\frac{\zeta}{ u_\tau}\right)\right] \\
&-\int_{\mathcal{C}}\nabla( u_\tau  v_\tau)\cdot\left[A\nabla\log \left(\frac{\zeta}{ u_\tau}\right)\right]\\
=&-\tau\int_{\mathcal{C}}  u_\tau  v_\tau\partial_t\log \left(\frac{\zeta}{ u_\tau}\right)+2\int_{\mathcal{C}} u_\tau v_\tau (\nabla\log  u_\tau)\cdot\left[A \nabla\log\left(\frac{\zeta}{ u_\tau}\right)\right] \\
&-\int_{\mathcal{C}}\nabla( u_\tau  v_\tau)\cdot\left[A\nabla\log \left(\frac{\zeta}{ u_\tau}\right)\right]+\int_{\mathcal{C}} u_\tau v_\tau \left[\nabla\log\left(\frac{\zeta}{ u_\tau}\right)\right]\cdot\left[ A\nabla\log\left(\frac{\zeta}{ u_\tau}\right)\right].\\
\end{split}
\end{equation*}
As $u_\tau \log\left(\frac{\zeta}{u_\tau }\right)\in\mathbb{S}_b$, we choose $\varphi=u_\tau \log\left(\frac{\zeta}{u_\tau }\right)$ in (\ref{fr10}). By $\mathbf{D}J_\tau(u_\tau)\varphi=0$, we have
\begin{equation*}
\begin{split}
   J_\tau( u_\tau)-J_\tau(\zeta)&=-\mathbf{D}J_\tau( u_\tau)\varphi+\int_{\mathcal{C}} u_\tau v_\tau \left[\nabla\log\left(\frac{\zeta}{ u_\tau}\right)\right]\cdot\left[ A\nabla\log\left(\frac{\zeta}{ u_\tau}\right)\right] \\
   &=\int_{\mathcal{C}} u_\tau  v_\tau \left[\nabla\log\left(\frac{\zeta}{ u_\tau}\right)\right]\cdot\left[ A\nabla\log\left(\frac{\zeta}{ u_\tau}\right)\right],
   \end{split}
\end{equation*}
as desired.\\

\noindent {\bf Case 2: $b=1$.}
The Hopf Boundary Lemma implies that $\nabla u_\tau\cdot\mathbf{n}<0$ and $\nabla v_\tau\cdot\mathbf{n}<0$ on $\partial\Omega\times[0,1]$, and thus  $u_\tau,v_\tau\in \mathbb{S}^{0}_1$, so that  $J_\tau(u_\tau)$ and $J_\tau(v_\tau)$ are well defined.
  Fix $t\in[0,1]$. For any $\zeta\in\mathbb{S}^{0}_1$,  noting that $\nabla\zeta\cdot\mathbf{n}<0$ on $\partial\Omega$,  we have
  $$\lim_{x\rightarrow x_0}\frac{u_\tau v_\tau}{\zeta}=\lim_{x\rightarrow x_0}\frac{ v_\tau\nabla u_\tau\cdot\mathbf{n}+ u_\tau\nabla v_\tau\cdot\mathbf{n}}{\nabla\zeta\cdot\mathbf{n}}=0,\,\,\forall x_0\in\partial\Omega.$$
  Hence, it is easy to see that $\int_{\partial\Omega}u_\tau v_\tau\left[A\nabla\log\zeta\right]\cdot\mathbf{n}=0$ in  (\ref{fr9}).

Similarly as in Case 1, we can show that the principal eigenfunction $u_\tau$ is still a critical point of $J_\tau$,
 i.e., $\mathbf{D}J(u_\tau)\varphi=0$ for all $\varphi\in\mathbb{S}_1$. Based on this fact,  formula (\ref{fr8})
  can be proved by a similar argument as in Case 1. The  proof of Lemma \ref{Flem2.1} is thus complete.
\end{proof}
With the help of Lemma \ref{Flem2.1}, we are in a position to prove Theorem \ref{FR_thm 1}.\\

$\mathbf{Proof~ of ~Theorem ~\ref{FR_thm 1}}$:  We substitute
 $u=u_\tau$ into (\ref{fr21}) and differentiate the resulting equation with respect to $\tau$. Denoting $\frac{\partial u_\tau}{\partial\tau}=u'_\tau$ for  brevity, we  obtain
 \begin{equation*}
 \begin{cases}
\begin{array}{ll}
\smallskip
 \partial_t u_{\tau}+\tau \partial_t u'_{\tau}-\mathrm{div}\left[A\nabla u'_\tau\right]+V u'_\tau=\lambda(\tau) u'_\tau+\frac{\partial\lambda}{\partial\tau} u_\tau,&x\in\Omega,t\in[0,1],\\
 \smallskip
 bu'_\tau+(1-b)\left[A\nabla  u'_\tau\right]\cdot\mathbf{n}=0,&x\in\partial\Omega,t\in[0,1],\\
   u'_\tau(x,0)= u'_\tau(x,1),&x\in\Omega.
  \end{array}
 \end{cases}
 \end{equation*}
We multiply the above equation by $v_\tau$ and integrate the result over $\mathcal{C}$. Together with the facts of  $ L^{*}_{\tau} v_\tau =\lambda(\tau)v_\tau $ and the normalization $\int_{\mathcal{C}}u_\tau v_\tau=1$, we find that
\begin{equation*}
    \frac{\partial\lambda}{\partial\tau}=\int_{\mathcal{C}} v_\tau\partial_tu_{\tau}.
\end{equation*}
Recalling  the definitions of $ L_{\tau}$, $ L^{*}_{\tau}$ and $J_\tau$, we further derive
\begin{equation*}
\begin{split}
  \int_{\mathcal{C}} v_\tau\partial_t(u_{\tau})=&\frac{1}{2\tau}\int_{\mathcal{C}}v_\tau (L_{\tau}- L^{*}_{\tau}) u_\tau\\
   =&\frac{1}{2\tau}\left[\int_{\mathcal{C}}v_\tau  L_{\tau} u_\tau-\int_{\mathcal{C}}u_\tau L_{\tau} v_\tau\right]\\
    =&\frac{1}{2\tau}\Big[J_\tau(u_\tau)-J_\tau(v_\tau)\Big].
\end{split}
\end{equation*}
In view of $v_\tau\in\mathbb{S}^{0}_b$,  Lemma \ref{Flem2.1} implies
\begin{equation}\label{Liukey}
\begin{split}
   \frac{\partial\lambda}{\partial\tau}
    &=\frac{1}{2\tau}\int_{\mathcal{C}}u_\tau v_\tau \left[\nabla\log\left(\frac{v_\tau}{u_\tau}\right)\right]\cdot\left[ A\nabla\log\left(\frac{v_\tau}{u_\tau}\right)\right],
\end{split}
\end{equation}
which shows that $\frac{\partial\lambda}{\partial\tau}\geq0$ for all $\tau>0$.

It remains to prove Parts (i) and (ii) in Theorem \ref{FR_thm 1}. In what follows we assume that $A(x,t)=\hat{A}(x)$. When $V(x,t)=\hat{V}(x)+g(t)$ for some $1$-periodic function $g(t)$,
we set $U_\tau(x,t)=e^{\left(\frac{1}{\tau}\int_{0}^{1}g(s)\mathrm{d}s\right)}u_\tau$. Then  $U_\tau$ is $1$-periodic and solves
\begin{equation*}
 \begin{cases}
\begin{array}{ll}
\smallskip
\tau\partial_t U_{\tau}-\mathrm{div}[\hat{A}\nabla U_\tau]+\hat{V}(x)U_\tau=\lambda(\tau)U_\tau,&x\in\Omega,t\in[0,1],\\
\smallskip
 bU_\tau +(1-b)[\hat{A}\nabla U_\tau]\cdot\mathbf{n}=0,&x\in\partial\Omega,t\in[0,1],\\
  U_\tau(x,0)=U_\tau(x,1), &x\in\Omega.
  \end{array}
 \end{cases}
 \end{equation*}
Observe that $\hat{A}$ and $\hat{V}$ are independent of $t$. By the uniqueness of principal eigenfunction (up to multiplication by a constant), it is clear that $\lambda(\tau)$ is constant for $\tau>0$. This proves Part (i) in Theorem \ref{FR_thm 1}.

Finally,  we  show $\frac{\partial\lambda}{\partial\tau}>0$ for all $\tau>0$ if  $V$ does not take the form of $V=\hat{V}(x)+g(t)$. Suppose to the contrary that there exists some $\tau_0>0$ such that $\frac{\partial\lambda}{\partial\tau}(\tau_0)=0$. According to  (\ref{Liukey}), we have $u_{\tau_0}=c(t)v_{\tau_0}$ for some $1$-periodic function $c(t)>0$.  Substituting $u_{\tau_0}=c(t)v_{\tau_0}$ into  $L_{\tau_0}u_{\tau_0}=\lambda(\tau_0)u_{\tau_0}$ and using $ L^{*}_{\tau_0}v_{\tau_0}=\lambda(\tau_0)v_{\tau_0}$, we can deduce
$$ c'(t)v_{\tau_0}+2 c(t)\partial_tv_{\tau_0}=0.$$
It then follows $\partial_t\log v_{\tau_0}=-\frac{c'(t)}{2c(t)}$ in $\mathcal{C}$, which depends only on $t$. Hence, $v_{\tau_0}$ is of the form $v_{\tau_0}=X_{\tau_0}(x)T_{\tau_0}(t)$ with some  $1$-periodic function $T_{\tau_0}(t)>0$ in $[0,T]$ and function $X_{\tau_0}(x)>0$ in $\Omega$. Again using $ L^{*}_{\tau}v_{\tau_0}=\lambda(\tau_0)v_{\tau_0}$, we arrive at
$$-\tau_0\frac{T'_{\tau_0}(t)}{T_{\tau_0}(t)}-\frac{\mathrm{div}[\hat{A}X_{\tau_0}](x)}{X_{\tau_0}(x)}+V(x,t)=\lambda(\tau_0)\,\,\mathrm{in} \,\,\mathcal{C}.$$
Thus, it is necessary that  $V$ has the form of $V=\hat{V}(x)+g(t)$, contradicting the previous assumption. The proof of Theorem \ref{FR_thm 1} is now complete.
\qed

\subsection{The case $A(x,t)=DI_{N\times N}$ and $\nabla m(x,t)=\nabla\hat{m}(x)$}
We now prove Theorem \ref{FR_thm 1.1}.  Under our assumption, problem (\ref{fr1}) reduces to the following:
 \begin{equation}\label{fr23}
\begin{cases}
\begin{array}{ll}
\tau\partial_{t}u-D\Delta u-\nabla \hat{m}(x)\cdot\nabla u+V(x,t)u=\lambda(\tau)u, &x\in\Omega,t\in[0,1],\\
bu+(1-b)D\nabla u\cdot\mathbf{n}=0, &x\in\partial\Omega,t\in[0,1],\\
u(x,0)=u(x,1), &x\in\Omega.
\end{array}
\end{cases}
\end{equation}

$\mathbf{Proof~ of ~Theorem ~\ref{FR_thm 1.1}}$: The argument is similar to that of Theorem \ref{FR_thm 1} and hence we only give a sketch here.
 By the transformation $\varphi=e^{\frac{\hat{m}}{2D}}u$,  problem (\ref{fr23}) can be rewritten as
\begin{equation}\label{fr6}
 \begin{cases}
\begin{array}{ll}
\smallskip
\mathcal{L}_{\tau}\varphi:=\tau \partial_{t}\varphi-D\Delta\varphi+h(x,t)\varphi=\lambda(\tau)\varphi,&x\in\Omega,t\in[0,1],\\
\smallskip
  \left(b-\frac{1-b}{2}\nabla \hat{m}\cdot\mathbf{n}\right)\varphi+(1-b)D\nabla \varphi\cdot\mathbf{n}=0,&x\in\partial\Omega,t\in[0,1],\\
  \varphi(x,0)=\varphi(x,1),&x\in\Omega,
  \end{array}
 \end{cases}
 \end{equation}
 with $h=\frac{\Delta\hat{m}}{2}+\frac{|\nabla\hat{m}|^2}{4D}+V(x,t)$. Denote by $\widetilde{u}_\tau>0$ the eigenfunction of problem (\ref{fr23}) corresponding to $\lambda(\tau)$. Then $\varphi_\tau=e^{\frac{\hat{m}}{2D}}\widetilde{u}_\tau$ (normalized by $\int_{\mathcal{C}}\varphi^2_\tau=1$) solves (\ref{fr6}). We also consider the adjoint problem to (\ref{fr6}):
\begin{equation}\label{fr113}
 \begin{cases}
\begin{array}{ll}
\smallskip
\mathcal{L}^{*}_{\tau}\psi:=-\tau \partial_{t}\psi-D\Delta\psi+h(x,t)\psi=\lambda(\tau)\psi,&x\in\Omega,t\in[0,1],\\
\smallskip
  \left(b-\frac{1-b}{2}\nabla \hat{m}\cdot\mathbf{n}\right)\psi+(1-b)D\nabla \psi\cdot\mathbf{n}=0,&x\in\partial\Omega,t\in[0,1],\\
  \psi(x,0)=\psi(x,1),&x\in\Omega.
  \end{array}
 \end{cases}
 \end{equation}
 Choose $\psi_\tau>0$ to be the principal eigenfunction of \eqref{fr113},  normalized by $\int_{\mathcal{C}}\varphi_\tau \psi_\tau=1$. Similarly as in Subsection \ref{S3.1}, we  introduce the functional
 \begin{equation*}
   \widetilde{J}_\tau(\zeta)=\int_{\mathcal{C}}\varphi_\tau \psi_\tau\left(\frac{\mathcal{ L}_{\tau}\zeta}{\zeta}\right)\mathrm{d}x\mathrm{d}t,
 \end{equation*}
 which is well defined on the cone
\begin{equation*}
\widetilde{\mathbb{S}}^{0}_b=
 \begin{cases}
 \begin{split}
 &\{\zeta \in \widetilde{\mathbb{S}}_b:  \zeta>0 ~\mathrm{in}~\overline{\mathcal{C}}\}, &0\leq b<1,\\
 &\{\zeta \in \widetilde{\mathbb{S}}_1:  \zeta>0 ~\mathrm{in}~\mathcal{C}, \nabla\zeta\cdot\mathbf{n}<0~\mathrm{on} ~\partial\Omega \times[0,1]\},&b=1.
  \end{split}
 \end{cases}
 \end{equation*}
Here $\widetilde{\mathbb{S}}_b$ is defined by
 \begin{equation*}
\widetilde{\mathbb{S}}_b=\left\{
\begin{array}{ll}
\medskip
\zeta \in  C^{2,1}(\mathcal{C})\cap C^{1,1}(\overline{\mathcal{C}}): &\zeta(x,0)=\zeta(x,T)~\mathrm{in}~\Omega, \\
 &\left(b-\frac{1-b}{2}\nabla \hat{m}\cdot\mathbf{n}\right)\zeta+(1-b)D\nabla \zeta\cdot\mathbf{n}=0 ~\mathrm{on} ~\partial\Omega \times[0,1]
\end{array}\right\}.
 \end{equation*}
Proceeding similarly  as in the proof of  Lemma \ref{Flem2.1}, one can check that $\widetilde{J}_\tau$ satisfies the property:
\begin{equation*}
   \widetilde{J}_\tau(\varphi_\tau)-\widetilde{J}_\tau(\zeta)=D\int_{\mathcal{C}}\varphi_\tau \psi_\tau \left|\nabla\log\left(\frac{\zeta}{\varphi_\tau}\right)\right|^2,\,\, \forall\zeta\in\widetilde{\mathbb{S}}^{0}_b.
\end{equation*}
Then the same analysis as in Theorem \ref{FR_thm 1} enables one to conclude  Theorem \ref{FR_thm 1.1}.
\qed

\section{Non-monotonicity of $\lambda(\tau)$ in the general case}\label{S3}
In this section, we will construct two examples to show  that the monotonicity of $\lambda(\tau)$ stated in Theorems \ref{FR_thm 1} and \ref{FR_thm 1.1}  may not hold for a general $m(x,t)$.
\begin{example}[A counterexample for Theorem \ref{FR_thm 1} when $\nabla m\neq0$]\label{Ex1}
Let $m(x,t)=x_1$, $A(x,t)=a(t)$, $V(x,t)=\frac{-x_1a'(t)}{2a^2(t)}$, and  $b=0$ in problem (\ref{fr1}). Here $a(t)\in C^1([0,1])$ is some $1$-periodic positive function satisfying $a'(t)\not\equiv0$. Now, consider the eigenvalue problem
\begin{equation}\label{fr15}
\begin{cases}
\begin{array}{ll}
\smallskip
\tau\partial_{t}u-a(t)\Delta u-\partial_{x_1}u+\frac{-x_1a'(t)}{2a^2(t)}u=\lambda(\tau)u,&x\in\Omega, t\in[0,1],\\
\smallskip
\nabla u\cdot\mathbf{n}=0,&x\in\partial\Omega,t\in[0,1],\\
u(x,0)=u(x,1), &x\in\Omega.
\end{array}
\end{cases}
\end{equation}
Since  $\hat{V}(x)=0$, Part (ii) in Theorem \ref{FR_thm 2} concludes that
$$\lambda(\infty)=\lim_{\tau\rightarrow\infty}\lambda(\tau)=0.$$
Let $u_\tau>0$ be the principal eigenfunction of (\ref{fr15}) and introduce $\varphi_\tau=e^{\frac{x_1}{2a(t)}}u_\tau$. By direct calculation, we have
\begin{equation*}
 \begin{cases}
\begin{array}{ll}
\smallskip
\tau \partial_t\varphi_{\tau}-a(t)\Delta \varphi_{\tau}+\left[\frac{\tau x_1a'(t)}{2a^2(t)}+\frac{1}{4a(t)}-\frac{x_1a'(t)}{2a^2(t)}\right]\varphi_\tau=\lambda(\tau)\varphi_\tau,&x\in\Omega, t\in[0,1],\\
\smallskip
  2a(t)\nabla \varphi_{\tau}\cdot\mathbf{n}-\varphi_{\tau}\mathbf{n}_{x_1}=0,&~~x\in\partial\Omega, t\in[0,1],\\
  \varphi_\tau(x,0)=\varphi_\tau(x,1),&x\in\Omega.
  \end{array}
 \end{cases}
 \end{equation*}
Taking $\tau=1$, we obtain
\begin{equation*}
 \begin{cases}
\begin{array}{ll}
\smallskip
\partial_t\varphi_{1}-a(t)\Delta \varphi_{1}+\frac{1}{4a(t)}\varphi_1=\lambda(1)\varphi_1,&x\in\Omega, t\in[0,1],\\
\smallskip
   2a(t)\nabla \varphi_{1}\cdot\mathbf{n}-\varphi_{1}\mathbf{n}_{x_1}=0,&x\in\partial\Omega, t\in[0,1],\\
   \varphi_1(x,0)=\varphi_1(x,1),&x\in\Omega.
  \end{array}
 \end{cases}
 \end{equation*}
We multiply the above equation by $\varphi_1$ and integrate the resulting equation over $\mathcal{C}=\Omega\times (0,1)$. Then it follows from the boundary condition of $\varphi_1$ that
\begin{equation*}
\begin{split}
  \lambda(1)\int_{\mathcal{C}}\varphi_{1}^{2}&=-\frac{1}{2}\int_{0}^{1}\!\!\!\int_{\partial\Omega}\varphi^2_{1}\,\mathbf{n}_{x_1} +\int_{\mathcal{C}}a(t)|\nabla \varphi_{1}|^{2}+\int_{\mathcal{C}}\frac{\varphi_{1}^{2}}{4a(t)}\\
   &=-\int_{\mathcal{C}}\varphi_1\partial_{x_1}\varphi_{1} +\int_{\mathcal{C}}a(t)|\nabla \varphi_{1}|^{2}+\int_{\mathcal{C}}\frac{\varphi_{1}^{2}}{4a(t)}\\
   &=\int_{\mathcal{C}}\bigg[\frac{\varphi_1}{2\sqrt{a(t)}}-\sqrt{a(t)}\partial_{x_1}\varphi_{1}\bigg]^2 +\int_{\mathcal{C}}a(t)\sum_{i=2}^{N}(\partial_{x_i} \varphi_{1})^{2}\\
   &\geq 0.
\end{split}
\end{equation*}
If $\lambda(1)=0$, it is easily seen that $\partial_{x_1}\varphi_{1}=\frac{\varphi_1}{2a(t)}$ and $\partial_{x_i} \varphi_{1}=0$ for all $2\leq i\leq N$, whence $\varphi_{1}=Ce^{\frac{x_1}{2a(t)}}$ with some constant $C>0$. In view of $\varphi_1=e^{\frac{x_1}{2a(t)}}u_1$, we arrive at $u_1=C$. Substituting $u_1=C$ in problem $(\ref{fr15})$ gives a contradiction to $a'(t)\not\equiv0$. Thus  $\lambda(1)>0=\lambda(\infty)$ and in turn $\lambda(\tau)$ is not non-decreasing in $\tau$.
\end{example}

\begin{example}[A counterexample for Theorem \ref{FR_thm 1.1} when $A=I_{N\times N}$ but $\nabla m(x,t)\neq\nabla \hat{m}(x)$]\label{Ex2}
Let $\Omega=(0,2\pi)$,  $m(x,t)=\cos x\sin t$, $D=1$, $V(x,t)=\frac{1}{2}\cos x(\sin t+\cos t)$, and $b=0$ in problem (\ref{fr23}). Thus we are led to the following problem:
\begin{equation}\label{fr14}
 \begin{cases}
\begin{array}{ll}
\smallskip
\tau\partial_{t}u-\partial_{xx}u+\sin x\sin t\partial_xu+\frac{1}{2}\cos x(\sin t+\cos t)u=\lambda(\tau)u,&x\in(0,2\pi),t\in[0,2\pi],\\
\partial_{x}u(0,t)=\partial_{x}u(2\pi,t)=0,&t\in[0,2\pi],\\
u(x,0)=u(x,2\pi),&x\in(0,2\pi).
  \end{array}
 \end{cases}
 \end{equation}
Clearly, $m(x,t)$ does not satisfy the assumption in  Theorem \ref{FR_thm 1.1}. Let $\lambda(\tau)$ be the principal eigenvalue of (\ref{fr14}) and $u_\tau>0$ be the corresponding eigenfunction. By Part (ii) in Theorem \ref{FR_thm 2}, we  infer that
$$\lambda(\infty)=\lim_{\tau\rightarrow\infty}\lambda(\tau)=0.$$
Set $\varphi_\tau=e^{\frac{1}{2}\cos x\sin t}u_\tau$ and $\tau=1$. Then $(\lambda(1), \varphi_1)$ solves
  \begin{equation*}
 \begin{cases}
\begin{array}{ll}
\smallskip
\partial_t\varphi_{1}-\partial_{xx}\varphi_{1}+\frac{\sin^2x\sin^2t}{4}\varphi_1=\lambda(1)\varphi_1,&x\in(0,2\pi),t\in[0,2\pi],\\
 \smallskip
\partial_{x}\varphi_{1}(0,t)= \partial_{x}\varphi_{1}(2\pi,t)=0,&t\in[0,2\pi],\\
\varphi_1(x,0)=\varphi_1(x,2\pi),&x\in(0,2\pi).
  \end{array}
 \end{cases}
 \end{equation*}
Multiplying the above equation by $\varphi_1$ and integrating the resulting equation over $\mathcal{C}=(0,2\pi)\times (0,2\pi)$, we arrive at
$$\lambda(1)\int_{\mathcal{C}}\varphi_{1}^{2}=\int_{\mathcal{C}}\left(\partial_x\varphi_{1}\right)^{2}+\int_{\mathcal{C}}\left[\frac{\sin^2x\sin^2t}{4}\right]\varphi_{1}^{2},$$
which implies immediately that $\lambda(1)>0=\lambda(\infty)$. Hence $\lambda(\tau)$ is not non-decreasing in $\tau$.\\
\end{example}

\section{ Asymptotic behaviors of $\lambda(\tau)$}\label{S4}
The goal of this section is to prove Theorem \ref{FR_thm 2}.\\

$\mathbf{Proof~of~Theorem~ \ref{FR_thm 2}}$. As already noted, the proof of Part (ii) in Theorem \ref{FR_thm 2} has been carried out by \cite{Nadin2009} for the space-time periodic case, which can be easily adapted to the present setting. We thus omit the details here and  refer the interested reader to Lemma 3.10 in \cite{Nadin2009}.

To prove Part (i) in Theorem \ref{FR_thm 2}, denote by $u_0(x,t)$  the principal eigenfunction of (\ref{fr2}) corresponding to $\lambda^0(t)$ for fixed $t$. For any $x\in\overline{\Omega}$, it is easy to see that  $u_0(x,\cdot), \nabla u_0(x,\cdot)\in C^{1}([0,1])$ and $u_0(x,t+1)=u_0(x,t)$.

Define $\underline{u}=\rho(t)u_0$ for some $1$-periodic function $\rho(t)>0$. Given arbitrary $\epsilon>0$, it is desirable to specify $\rho(t)$ such that 
\begin{equation}\label{fr4}
\left\{
\begin{array}{ll}
\smallskip
 \tau \partial_{t}\underline{u}-\mathrm{div}\left[A\nabla \underline{u}\right]-\nabla m\cdot\nabla\underline{u}+V\underline{u}\leq\left(\int_{0}^{1}\lambda^0(s)\mathrm{d}s+\epsilon\right)\underline{u},&x\in\Omega,t\in[0,1],\\
\smallskip
b\underline{u}+(1-b)\left[A\nabla \underline{u}\right] \cdot\mathbf{n}=0,&x\in\partial\Omega,t\in[0,1],\\
\smallskip
\underline{u}(x,0)=\underline{u}(x,1),&x\in\Omega,
\end{array}
\right.
\end{equation}
provided that  $\tau>0$  is sufficiently small.  Then a direct application of Proposition 2.1 of \cite{PZ2015} gives  
\begin{equation}\label{fr_sup}
 \limsup_{\tau\rightarrow0}\lambda(\tau)\leq \int_{0}^{1}\lambda^0(s)\mathrm{d}s+\epsilon.
\end{equation}


To verify (\ref{fr4}), substitute $\underline{u}=\rho(t)u_0$ into (\ref{fr4}). Direct calculation shows that $\rho(t)$ should satisfy
\begin{equation*}
\begin{array}{l}
\tau\rho'u_0+\tau\rho\partial_tu_0+\lambda^0(t)\rho u_0\leq\left(\int_{0}^{1}\lambda^0(s)\mathrm{d}s+\epsilon\right)\rho u_0.
\end{array}
\end{equation*}

With this in mind, we now claim that $\frac{\partial_tu_0}{u_0}$ is bounded in $\overline{\Omega}\times[0,1]$. Indeed, the claim is obvious for $0\leq b<1$  because of $u_0>0$ in $\overline{\mathcal{C}}$. For $b=1$ (i.e., zero Dirichlet boundary condition), it suffices to prove $\lim\limits_{x\rightarrow\partial\Omega}\frac{\partial_tu_0}{u_0}$ is bounded. The well-known Hopf Boundary Lemma for parabolic equations (see e.g. Proposition 13.3 of \cite{Hess})  implies that  $\nabla u_0\cdot\mathbf{n}<0$ on $\partial\Omega\times[0,1]$.
Fix any $ x_0\in\partial\Omega$. Thanks to $u_0(x_0,t)=\partial_tu_0( x_0,t)=0$, the claim follows from
\begin{equation*}
\begin{array}{l}
\lim\limits_{x\rightarrow x_0}\frac{\partial_tu_0}{u_0}(x,t)=\frac{\nabla\partial_tu_0\cdot\mathbf{n}}{\nabla u_0\cdot\mathbf{n}}(x_0,t).
\end{array}
\end{equation*}

Consequently, we can choose $\tau$ so small that $\tau|\partial_tu_0|\leq\epsilon u_0$. Then define $\rho(t)$ by
\begin{equation*}
\begin{array}{l}
\tau\rho'=\left(\int_{0}^{1}\lambda^0(s)\mathrm{d}s-\lambda^0(t)\right)\rho,
\end{array}
\end{equation*}
i.e.,
$\rho(t)=\exp\left[\frac{1}{\tau}\left(t\int_{0}^{1}\lambda^0(s)\mathrm{d}s-\int_{0}^{t}\lambda^0(s)\mathrm{d}s\right)\right]$. We can check readily  that $\rho(t)=\rho(t+1)$ and  $\underline{u}$ verifies (\ref{fr4}), whence (\ref{fr_sup}) holds.

In a similar manner, we may construct $\overline{u}$ such that
\begin{equation*}
\left\{
\begin{array}{ll}
\smallskip
 \tau \partial_{t}\overline{u}-\mathrm{div}\left[A\nabla \overline{u}\right]-\nabla m\cdot\nabla\overline{u}+V\overline{u}\geq\left(\int_{0}^{1}\lambda^0(s)\mathrm{d}s-\epsilon\right)\overline{u},&x\in\Omega,t\in[0,1],\\
\smallskip
b\overline{u}+(1-b)\left[A\nabla \overline{u}\right] \cdot\mathbf{n}=0,&x\in\partial\Omega,t\in[0,1],\\
\smallskip
\overline{u}(x,0)=\overline{u}(x,1),&x\in\Omega,
\end{array}
\right.
\end{equation*}
 for sufficiently small $\tau>0$, which combined with Proposition 2.1 of \cite{PZ2015} leads to
 \begin{equation}\label{fr_inf}
\liminf_{\tau\rightarrow0}\lambda(\tau)\geq\int_{0}^{1}\lambda^0(s)\mathrm{d}s-\epsilon.
\end{equation}

 Part (i) in Theorem \ref{FR_thm 2} follows from  (\ref{fr_sup}) and (\ref{fr_inf}).
\qed

\bigskip
\noindent{\bf Acknowledgments.} SL was partially supported by the NSFC grant No. 11571364.
YL  was partially supported by the NSF grant DMS-1411176. RP was partially supported by the National
Science Foundation (NSF) of China (grant nos. 11671175 and 11571200), the
Priority Academic Program Development of Jiangsu Higher Education Institutions,
the Top-notch Academic Programs Project of Jiangsu Higher Education
Institutions (grant no. PPZY2015A013), and the Qing Lan Project of Jiangsu
Province. MZ was partially supported by the Australian Research
Council (grant no. DE170101410).



\begin{thebibliography}{00}

\bibitem{Berestycki2005}  Berestycki, H., Hamel, F.,    Nadirashvili, N.: Elliptic eigenvalue problems with large drift and applications to nonlinear propagation phenomena,  Commun. Math. Phys., 253 (2005) 451-480.
\bibitem{CC2003}
Cantrell, R.S.,   Cosner, C.: 
Spatial Ecology via Reaction-Diffusion Equations. Series in
Mathematical and Computational Biology, John Wiley and Sons, Chichester, UK, 2003.
\bibitem{CL2008}  Chen, X.F., Lou, Y.: Principal eigenvalue and eigenfunctions of an elliptic operator with large advection and its application to a competition model, Indiana Univ. Math. J., 57   (2008) 627-658.

\bibitem{CL2012}  Chen, X.F., Lou, Y.:   Effects of diffusion and advection on the smallesteigenvalue of an elliptic operators and their applications, Indiana Univ. Math J., 60 (2012) 45-80.


\bibitem{Hess}
Hess, P.: Periodic-parabolic Boundary Value Problems and
Positivity, Pitman Res., Notes in Mathematics  247, Longman
Sci. Tech., Harlow, 1991.

\bibitem{Hutson2000} Hutson, V., Shen, W., Vickers, G.T.: Estimates for the principal spectrum point for certain time-dependent
parabolic operators, Proc. A.M.S., 129 (2000) 1669-1679.

\bibitem{Hutson2001}Hutson, V., Michaikow, K., Pol\'{a}\v{c}ik, P.: The evolution of dispersal rates in a heterogeneous time-periodic
environment, J. Math. Biol., 43 (2001) 501-533.


\bibitem{LL2018} Liu, S., Lou, Y.: A functional approach towards eigenvalue problems associated with incompressible flow, Submitted  (2018).



\bibitem{Nadin2009} Nadin, G.: The principal eigenvalue of a space-time periodic parabolic operator, Ann. Mat. Pur. Appl., 188 (2009) 269-295.

\bibitem{Ni2011}Ni, W.M.:  The Mathematics of Diffusion, CBMS-NSF Regional Conf. Ser. in Appl. Math. 82, SIAM,
Philadelphia, 2011.


\bibitem{PZZ2018}Peng, R., Zhang, G., Zhou, M.: Asymptotic behavior of the principal eigenvalue of a linear second order elliptic operator with small/large diffusion coefficient and its application, preprint (arXiv:1809.10815).

\bibitem{PZ2015}Peng, R., Zhao, X.Q.: Effects of diffusion and advection on the principal eigenvalue of a periodic-parabolic problem with applications, Calc. Var. Partial Diff., 54 (2015) 1611-1642.

\bibitem{PZ2017}Peng, R., Zhou, M.: Effects of large degenerate advection and boundary conditions on the principal eigenvalue and its eigenfunction of a linear second order elliptic operator, Indiana Univ. Math J.,  67 (2018) 2523-2568.



\end{thebibliography}
\end{document}